\documentclass[reqno,12pt]{amsart}
\usepackage[utf8]{inputenc}
\usepackage[left=1.2in, right=1.2in, top=1in]{geometry}

\usepackage{esint}
\usepackage{amsmath,amssymb,amsthm,mathrsfs,color,times,textcomp,verbatim,mathtools}
\allowdisplaybreaks
\usepackage{xcolor}
\usepackage[colorlinks=true]{hyperref}
\hypersetup{urlcolor=blue, citecolor=red, linkcolor=blue}
\usepackage[square,numbers]{natbib}

\numberwithin{equation}{section}
\theoremstyle{plain} 
\newtheorem{theorem}{Theorem}[section]
\newtheorem*{theorem*}{Theorem}

\newtheorem{lemma}[theorem]{Lemma}

\newtheorem{corollary}[theorem]{Corollary}

\usepackage{graphicx}

\newtheorem{remark}[theorem]{Remark}

\title[Sobolev trace inequality on balls]{Remarks on extremals of sharp Sobolev trace inequalities on the unit balls}

\author{Cheikh Birahim NDIAYE}
\address{Department of Mathematics Howard University Annex 3, Graduate School of Arts and Sciences DC 20059 Washington, USA.}
\email{cheikh.ndiaye@howard.edu}
\author{Liming Sun}
\address{Academy of Mathematics and Systems Science, the Chinese Academy of Sciences, Beijing 100190, China.}
\email{lmsun@amss.ac.cn}

 \date{\today \,(Last Typeset)}

\def\R{\mathbb{R}}

\begin{document}

\maketitle

\begin{abstract}
    We show explicit forms for extremals of some fourth order sharp trace inequalities on the unit balls recently proved by Ache-Chang. We also give a classification result of the bi-harmonic equation on $\mathbb{R}^4_+$ with some conformally covariant boundary conditions. 
\end{abstract}

\section{Introduction}

Let $\mathbb{B}^{n+1}\subset \mathbb{R}^{n+1}$, $n\ge 1$, be the unit ball with boundary $\partial \mathbb{B}^{n+1}=\mathbb{S}^n$.  Recall the following two Sobolev trace inequalities: For any $f \in C^{\infty}\left(\mathbb{S}^{n}\right)$ and  $v$ being a smooth extension of $f$ to $\mathbb{B}^{n+1}$, there hold
 \begin{itemize}
     \item[-] If $n=1$, then 
     \begin{align}\label{lebedev-main}
\log \left(\frac{1}{2 \pi} \oint_{\mathbb{S}^{1}} e^{f} d \sigma\right) \leqslant \frac{1}{4 \pi} \int_{\mathbb{B}^{2}}|\nabla v|^{2} d x+\frac{1}{2 \pi} \oint_{\mathbb{S}^{1}} f d \sigma.
\end{align}
Moreover, the equality holds if and only if $\Delta v=0$ and  $f=c-\log \left|1-\left\langle z_{0}, \xi\right\rangle\right|$,
where $c$ is a constant, $\xi \in \mathbb{S}^{1}$, and $z_{0}$ is some fixed point in the interior of $\mathbb{B}^{2}$.
\item[-] If $n>1$, then 
\begin{align}\label{LM>2}
\frac{\Gamma\left(\frac{n+1}{2}\right)}{\Gamma\left(\frac{n-1}{2}\right)} |\mathbb{S}^n|^{1 / n}\left(\oint_{\mathbb{S}^{n}}|f|^{\frac{2 n}{n-1}} d \sigma\right)^{\frac{n-1}{n}} \leqslant \int_{\mathbb{B}^{n+1}}|\nabla v|^{2} d x+\frac{n-1}{2}\oint_{\mathbb{S}^{n}}|f|^{2} d \sigma.
\end{align}
Moreover, the equality holds if and only if $\Delta v=0$ and  $f=c|1-\langle z_0,\xi\rangle|^{-(n-1)/2}$, where $c$ is a constant, $\xi\in \mathbb{S}^n$ and $z_0\in \mathbb{B}^{n+1}$.
 \end{itemize}

The first one \eqref{lebedev-main} was proved by Lebedev-Milin \cite{lebedev1951coefficients} and Osgood-Phillips-Sarnak \cite{osgood1988extremals}. The second one \eqref{LM>2} was proved by \citet{lions1985concentration},  \citet{escobar1988sharp} and \citet{beckner1993sharp}. 

A natural question is what the extremal $v$ look like in the unit ball. It is not hard to find 
 the harmonic extension of $-\log|1-\langle z_0,\xi\rangle|$ on $\mathbb{B}^2$ is
\begin{align}\label{ext-2}
    v(\xi)=-\log\left|\frac{\xi}{|\xi|} -{|\xi|} \omega_0\right|^2+\log(1+|\omega_0|^2), \quad \xi\in \mathbb{B}^2
\end{align}
and the harmonic extension of $|1-\langle z_0,\xi\rangle|^{-\frac{n-1}{2}}$ on $\mathbb{B}^{n+1}$ with $n>1$ is 
\begin{align}\label{ext->2}
v(\xi)=(1+|\omega_0|^{2})^{n-1}\left|\frac{\xi}{|\xi|} -{|\xi|} \omega_0\right|^{1-n},\quad  \xi\in \mathbb{B}^{n+1}
\end{align}
where $w_0=\frac{z_0}{|z_0|^2}(1-\sqrt{1-|z_0|^2})\in \mathbb{B}^{n+1}$. To obtain \eqref{ext-2}, one can use the observation 
\begin{align}
    \log |\omega_0-\xi|^2=\log (1-2\omega_0\cdot \xi+|\omega_0|^2)=\log (1-\langle z_0,\xi\rangle)+\log (1+|\omega_0|^2)
\end{align}
because $z_0=2\omega_0/(1+|\omega_0|^2)$.
One notices that $\log |\omega_0-\xi|^2$ is a harmonic function with a pole in $\mathbb{B}^{2}$. Using Green's function to annihilate the singularity, we can find the explicit forms for the extremal $v$. This approach also works for \eqref{ext->2}.

\citet{ache2017sobolev} generalized the Lebedev-Milin inequality and its counterpart \eqref{LM>2} to ones of \textit{order four}. More precisely, 
let $f \in C^{\infty}\left(\mathbb{S}^{n}\right)$ and $v$ be a smooth extension of $f$ to the unit ball $\mathbb{B}^{n+1}$. Let $\eta$ be the outward-pointing unit normal to $\mathbb{S}^n$ and $\bar \nabla$ be the gradient in $\mathbb{S}^n$. Then we have the sharp trace inequalities. 
\begin{itemize}
    \item[-] If $n=3$, then 
    \begin{align}\label{ache-main=4}
\quad \log \left(\frac{1}{2 \pi^{2}} \oint_{\mathbb{S}^{3}} e^{3f} d \sigma\right) \leq \frac{3}{16 \pi^{2}} \int_{\mathbb{B}^{4}}\left(\Delta v\right)^{2} d x+\frac{3}{8 \pi^{2}} \oint_{\mathbb{S}^{3}}|\bar{\nabla} f|^{2} d \sigma+\frac{3}{2\pi^2}\oint_{\mathbb{S}^3}fd\sigma
\end{align}
for any $v$ satisfying the homogeneous Neumann boundary condition $\eta v|_{\mathbb{S}^3}=0$. Equality holds if and only if $f=c-\log \left|1-\left\langle z_{0}, \xi\right\rangle\right|$ where $c$ is a constant, $\xi \in \mathbb{S}^{3}, z_{0}$ is some point in $\mathbb{B}^{4}$ and $v$ satisfies that 
\begin{align}\label{bi-h-v1}
\begin{cases}
    \Delta^2 v=0& \text{in }\mathbb{B}^{4},\\
    \eta v=0&\text{on }\mathbb{S}^3,\\
    v=f&\text{on }\mathbb{S}^3.
\end{cases}
\end{align}

\item[-] If $n>3$, then 
\begin{align}\label{ache-main>5}
a_n\left(\oint_{\mathbb{S}^{n}}|f|^{\frac{2 n}{n-3}} d \sigma\right)^{\frac{n-3}{n}} \leqslant \int_{\mathbb{B}^{n+1}}|\Delta v|^{2} d x+2 \oint_{\mathbb{S}^{n}}|\bar{\nabla} f|^{2} d \sigma+b_{n} \oint_{\mathbb{S}^{n}}|f|^{2} d \sigma,
\end{align} 
for any $v$ satisfying the  Neumann boundary condition $\left.\eta v\right|_{\mathbb{S}^{n}}=-\frac{n-3}{2} f.$
Here $a_n=2 \frac{\Gamma\left(\frac{n+3}{2}\right)}{\Gamma\left(\frac{n-3}{2}\right)} |\mathbb{S}^n|^{3 / n}$ and $b_{n}=(n+1)(n-3) / 2$. Equality holds if and only if $f=c\left|1-\left\langle z_{0}, \xi\right\rangle\right|^{\frac{3-n}{2}}$ and $v$ satisfies that 
\begin{align}\label{bi-h-v2}
     \begin{cases}
    \Delta^2 v=0& \text{in }\mathbb{B}^{n+1},\\
    \eta v=-\frac{n-3}{2}v&\text{on }\mathbb{S}^n,\\
    v=f&\text{on }\mathbb{S}^n.
    \end{cases}
\end{align}
\end{itemize}

Similar to the second order case, we also want to know what the extremal functions of \eqref{ache-main=4} and \eqref{ache-main>5} look like in $\mathbb{B}^{n+1}$. We introduce the function $F:\overline{\mathbb{B}^{n+1}}\times \overline{\mathbb{B}^{n+1}}\to \mathbb{R}$ as
\begin{align}\label{def:F}
F(\xi,\omega)=\left|\frac{\xi}{|\xi|}-|\xi|\omega\right|.
\end{align}
\begin{theorem}\label{thm:ext>=4} 
Given any $z_0\in \mathbb{B}^{n+1}$, define $\omega_0=\frac{z_0}{|z_0|^2}(1-\sqrt{1-|z_0|^2})\in \mathbb{B}^{n+1}$. 
\begin{enumerate}
\item Let $v$ be a solution of \eqref{bi-h-v1} with  $f=-\log|1-\langle z_0,\xi\rangle|$ on $\mathbb{B}^4$. Then
\begin{align}
    v(\xi)=&\ -\log{F(\xi,\omega_0)^2}+\frac{(1-|\xi|^2)}{2}\left[\frac{1-|\omega_0|^2}{F(\xi,\omega_0)^2}-1\right]+\log (1+|\omega_0|^2).
\end{align}
\item  Let $v$ be a solution of \eqref{bi-h-v2} with $n>3$ and $f=|1-\langle z_0,\xi\rangle|^{(3-n)/2}$. Then
\begin{align}
      v(\xi)=\frac{\left({1+|\omega_0|^2}\right)^{\frac{n-3}{2}}}{F(\xi,\omega_0)^{n-3}}\left[1+\frac{(n-3)(1-|\omega_0|^2)(1-|\xi|^2)}{4F(\xi,\omega_0)^2}\right].
\end{align}
\end{enumerate}
\end{theorem}

Proving the above theorem is more complicated than that of the second-order case because we have one more boundary condition in the fourth-order case.  We reformulate the problem on $\mathbb{R}^{n+1}_+$ via the M\"obius transformation (see \eqref{def:S}) as the extremals of \eqref{ache-main=4} and \eqref{ache-main>5} satisfy simpler equations on $\mathbb{R}^{n+1}_+$.    In fact, when $n>3$, \citet{case2018boundary} and  \citet{ngo2020higher} prove that \eqref{ache-main>5} is equivalent to a sharp trace inequality on $\mathbb{R}^{n+1}_+$ whose extremal function satisfies 
\begin{align}\label{main-eq>5}
\begin{cases}
  \Delta^2u=0&\text{in}\quad\mathbb{R}^{n+1}_+,\\
    \partial_t\Delta u=cu^{\frac{n+3}{n-3}}&\text{on}\quad \partial \mathbb{R}^{n+1}_+,\\
    \partial_t u=0&\text{on}\quad \partial \mathbb{R}^{n+1}_+,
\end{cases}
\end{align}
for some constant $c>0$. Fortunately, \citet{sun2016classification} has found the definitive solution to the above equation. One can move everything back to the unit ball through a M\"obius transformation. 


In the case $n=3$, we do not find an equivalent formulation of \eqref{ache-main=4} on $\mathbb{R}^4_+$. However,  we notice the extremal of \eqref{ache-main=4} implies that $(\mathbb{B}^4,e^{2v}g^*)$ has vanishing $Q$-curvature, constant $T$-curvature and vanishing mean curvature.  Here $T$-curvature is defined by \citet{chang1997zeta} and $g^*$ is the \textit{adapted metric} on $\mathbb{B}^4$ (see \cite[Prop 2.2]{ache2017sobolev} and \cite{case2016fractional}). It is interesting that one has to use the adapted metric here. Via M\"obius transformation, these geometric conditions are equivalent to the existence of $u$ which satisfies the following equation\footnote{Here the coefficient $4$ in front of $e^{3u}$ is normalized for the convenience of \eqref{def:u}. The other solutions to  $\partial_t\Delta u=ce^{3u}$ differs from \eqref{def:u} by adding a constant.} 
\begin{align}\label{main-eq}
\begin{cases}
    \Delta^2u=0&\text{in}\quad\mathbb{R}^4_+,\\
    \partial_t\Delta u=4e^{3u}&\text{on}\quad \partial \mathbb{R}^4_+,\\
    \partial_t u=0&\text{on}\quad \partial \mathbb{R}^4_+.
\end{cases}
\end{align}

This equation is of a similar type to \eqref{main-eq>5} and has yet been studied in \cite{sun2016classification}. Here we continue to classify the solutions of this equation under the \textit{finite volume conditions} \eqref{finite-vol} and \eqref{finite-vol-in}. These are very natural geometric conditions. 
\begin{theorem}\label{thm:main-1}
Suppose that $u\in C^4(\overline{\R^4_+})$ satisfies \eqref{main-eq} and the following conditions.
\begin{align}
    (i)& \quad \int_{\mathbb{R}^3}e^{3u(x,0)}dx<\infty,\label{finite-vol}\\
    (ii)& \quad \int_{\mathbb{R}^4_+}e^{4u(x,t)}dxdt<\infty.\label{finite-vol-in}
\end{align}
    Then either $\lim_{x\to\infty }\bar\Delta u(x,0)$ exists and negative, here $\bar{\Delta}$ is the Laplacian w.r.t. $x\in \R^3$ only, or then there exist $a\in \R^3$, $\lambda>0$ and $c\leq 0$ such that  $u=u_{a,\lambda}(x,t)+ct^2$ where
\begin{align}\label{def:u}
    u_{a,\lambda}(x,t)=\log\left(\frac{2\lambda}{(\lambda+t)^2+|x-a|^2 }\right)+\frac{2t\lambda}{(\lambda+ t)^2+|x-a|^2}.
\end{align}
\end{theorem}

\begin{remark}
Conditions (i) and (ii) are sharp in the sense that if we remove both of them, then there are other solutions. For instance, $u=\frac23t^3$ satisfies \eqref{main-eq} but violates \eqref{finite-vol} and \eqref{finite-vol-in}. On the other hand $u=u_{a,\lambda}+ct^2$ for any $c>0$ satisfies \eqref{main-eq} and \eqref{finite-vol} but violates \eqref{finite-vol-in}. 
\end{remark}

As a byproduct of our arguments, we have the following corollary.
\begin{corollary}\label{corcla}
Suppose that $u\in C^4(\overline{\R^4_+})$ satisfies \eqref{finite-vol} and 
\begin{align}\label{u=int}
u(x, t)=\frac{1}{|\mathbb{S}^3|}\int_{\mathbb{R}^3}e^{3u(y, 0)}\log\frac{|y|^2}{|x-y|^2+t^2}dy
\end{align}
then there exists $a\in \mathbb{R}^3$ and $\lambda>0$ such that $u=u_{a, \lambda}$.
\end{corollary}

\begin{remark}
 Corollary \ref{corcla} and Theorem \ref{thm:ext>=4} are expected to be used to describe the asymptotic behavior of sequences of conformal metrics with prescribed $T$-curvature,  $Q=0$ and $H=0$  on the whole background manifold rather than just at the boundary as available results in the literature provides. 
 This is one of the motivations why we consider the extensions of extremals of Ache-Chang's inequality.
\end{remark}

The paper is organized as follows. In section 2, we first give some preliminary on a M\"obius transformation which maps the upper half space to the unit ball.  Subsection 2.1 is devoted to finding the extension of extremals on dimension five and above. The case of dimension four is studied in the subsequent subsection 2.2. In the last section, we prove the classification theorem about a bi-harmonic equation on $\mathbb{R}^4_+$ with some conformally covariant boundary conditions.

\section{Sobolev trace inequality of order four}

Recall that the following M\"obius transformation which maps the upper half space to  the unit ball.
\begin{align}\label{def:S}
\begin{split}
\mathcal{S}:\R^{n+1}_+=\{X=(x,t)\}&\mapsto \mathbb{B}^{n+1}=\{\xi=(\xi',\xi_{n+1})\}\\
X&\to\frac{2(X+e_{n+1})}{|X+e_{n+1}|^2}-e_{n+1}
\end{split}
\end{align}
where $e_{n+1}=(0,\cdots,0,1)\in \mathbb{R}^{n+1}$.
Conversely
\begin{align}\label{x=xi,t=xi4}
\mathcal{S}^{-1}(\xi)=\frac{2(\xi+e_{n+1})}{|\xi+e_{n+1}|^2}-e_{n+1}.
\end{align}
It is well-known that 
\begin{align}
    \mathcal{S}^*|d\xi|^2=\left(\frac{2}{|X+e_{n+1}|^2}\right)^2|dX|^2.
\end{align}

\begin{align}\label{|xi|2}
    |x|^2+t^2=\frac{-4\xi_{n+1}}{|\xi'|^2+(\xi_{n+1}+1)^2}+1,\quad |\xi|^2=\frac{-4t}{|x|^2+(t+1)^2}+1.
\end{align}
\begin{lemma}For any $(a,\lambda)\in \mathbb{R}^{n+1}_+$, the following identity holds for $\omega=\mathcal{S}(a,\lambda)\in \mathbb{B}^{n+1}$   
\begin{align}\label{tx2xi}
     \frac{\lambda}{|x-a|^2+|t+\lambda|^2}=\frac{(1-|\omega|^2)}{4}\ \frac{|\xi+e_{n+1}|^2}{F(\xi,\omega)^2}.
\end{align}
\end{lemma}
\begin{proof} 
Denote $A=|\xi'|^2+|\xi_{n+1}+1|^2$ for short. We plug in \eqref{x=xi,t=xi4} to the LHS and achieve 
\begin{align*}
    &\ |\lambda+t|^2+|x-a|^2=A^{-2}[(2(\xi_{n+1}+1)+(\lambda-1)A)^2+|2\xi'-aA|^2]\\
    =&\ \frac{1}{A^2}[((\lambda-1)^2+|a|^2)A^2+4(\xi_{n+1}+1)(\lambda-1)A-4Aa\cdot \xi'+4A]\\
    =&\ \frac{1}{A}[((\lambda-1)^2+|a|^2)|\xi|^2+2(\lambda^2+|a|^2-1)\xi_{n+1}-4a\cdot \xi'+(\lambda+1)^2+|a|^2]\\
    =&\ \frac{(\lambda+1)^2+|a|^2}{|\xi'|^2+|\xi_{n+1}+1|^2}[|\omega|^2|\xi|^2-2\omega\cdot \xi+1]\\
    =&\ \frac{(\lambda+1)^2+|a|^2}{|\xi'|^2+|\xi_{n+1}+1|^2}\left|\frac{\xi}{|\xi|}-|\xi|\omega\right|^2
\end{align*}
where $\omega=\mathcal{S}(a,\lambda)$ with
\begin{align}
    \omega'=\frac{2a}{(\lambda+1)^2+|a|^2},\quad \omega_{n+1}=\frac{2(\lambda+1)}{(\lambda+1)^2+|a|^2}-1.
\end{align}
We also used the following fact
\begin{align}
|\omega|^2=\frac{-4\lambda}{(\lambda+1)^2+|a|^2}+1.
\end{align}
Consequently
\begin{align}
    \frac{\lambda}{|\lambda+t|^2+|x-a|^2}=\frac{1-|\omega|^2}{4}\frac{|\xi'|^2+|\xi_{n+1}+1|^2}{F(\xi,\omega)^2}.
\end{align}
\end{proof}

\subsection{Ache-Chang inequality on dimension five and above}
In this subsection, we shall consider the case $n>3$. 

\begin{proof}[Proof of Theorem \ref{thm:ext>=4} Part (1)] Suppose $v$ is the extension. Obviously when $z_0=0$, $v$ will be a positive constant. Since $v$ depends continuously on the boundary value, we can suppose $|z_0|$ is small enough such that $v>0$ in $\mathbb{B}^{n+1}$.

We define
\begin{align}
U(x,t)=v(\mathcal{S}(x,t))\left(\frac{2}{|x|^2+(1+t)^2}\right)^{\frac{n-3}{2}}\label{w=vS>4}
\end{align}
Then it is easy to see that
\begin{align}
\int_{\mathbb{R}^{n}}|U(x,0)|^{\frac{2n}{n-3}}dx=\oint_{\mathbb{S}^{n}}|v|^{\frac{2n}{n-3}}d\sigma
\end{align}
and $\eta v=-\frac{n-3}{2}v$ is equivalent to $\partial_t U(x,0)=0$ for $\forall\ x\in \mathbb{R}^n$.
After some computation (for instance, see \citet[Eq. (4.9)]{ngo2020higher}\footnote{The M\"obius transformation in \cite{ngo2020higher} is different from ours (see \eqref{def:S}) by a negative sign in the $\xi_{n+1}$ coordinate. However, this difference does not affect the energy identity.}), we obtain
\begin{align}
\int_{\mathbb{R}^{n+1}_+}|\Delta U|^2dX=\int_{\mathbb{B}^{n+1}}|\Delta v|^{2} d x+2 \oint_{\mathbb{S}^{n}}|\bar{\nabla} f|^{2} d \sigma+b_n \oint_{\mathbb{S}^{n}}|f|^{2} d \sigma.
\end{align}
Consequently \eqref{ache-main>5} is equivalent to the sharp trace inequality
\begin{align}
a_n\left(\int_{\mathbb{R}^{n}}|U(x, 0)|^{\frac{2 n}{n-3}} d x\right)^{\frac{n-3}{n}} \leqslant \int_{\mathbb{R}_{+}^{n+1}}|\Delta U(x, t)|^{2} d x d t
\end{align}
for functions $U$ with $\partial_tU(x,0)=0$. The extremal function of the above equation satisfies
\begin{align}\label{E-L-5-R}
\begin{cases}
    \Delta^2 U=0& \text{in }\mathbb{R}^{n+1}_+,\\
    \partial_t\Delta U=cU^{\frac{n+3}{n-3}}&\text{on }\mathbb{R}^n,\\
    \partial_t U=0&\text{on }\mathbb{R}^n,
    \end{cases}
\end{align}
for some $c>0$.
Using \eqref{w=vS>4}, we know that $U$ is also positive. The positive solutions to the above equations have been studied by \cite{sun2016classification}. It follows from \eqref{w=vS>4} that $U(x,t)=o(|x|^2+t^2)$. Therefore, one can apply \cite[Rmk 1.2]{sun2016classification}  to achieve that there exists $\lambda>0$, $a\in \mathbb{R}^3$ and 
\begin{align}
U(x,t)=c\left(\frac{\lambda}{(\lambda+t)^2+|x-a|^2}\right)^{\frac{n-3}{2}}\left[1+\frac{(n-3)t\lambda}{(\lambda+t)^2+|x-a|^2}\right]
\end{align}
for some constant $c>0$. Now we plug in the above equation to 
\begin{align}
v(\xi)=w(\mathcal{S}^{-1}(\xi))\left(\frac{2}{|\xi'|^2+(1+\xi_{n+1})^2}\right)^{\frac{n-3}{2}}\label{v=wS>4}
\end{align}
and 
using \eqref{tx2xi} and \eqref{x=xi,t=xi4} to obtain that 
\begin{align}
    v(\xi)=c\frac{\left({1-|\omega_0|^2}\right)^{\frac{n-3}{2}}}{F(\xi,\omega_0)^{n-3}}\left[1+\frac{(n-3)(1-|\omega_0|^2)(1-|\xi|^2)}{4F(\xi,\omega_0)^2}\right]
\end{align}
for some $c>0$. Here $\omega_0=\mathcal{S}(a,\lambda)$ and $F$ is defined in \eqref{def:F}. One can determine $c(1-|\omega_0|^2)^{(n-3)/2}=(1+|\omega_0|^2)^{(n-3)/2}$ using $v(\xi)=|1-\langle z_0,\xi\rangle|^{(3-n)/2}$ on $\mathbb{S}^n$.

It follows that $v$ is strictly positive and depends on $z_0$ smoothly. Since the above proof only requires $v$ to be positive, using the continuity method, one can prove that it holds for any $z_0\in \mathbb{B}^{n+1}$.


\end{proof}

\subsection{Ache-Chang inequality on dimension four}
Recall the Paneitz operator defined on a smooth compact Riemannian manifold $(X^{n+1},g)$ for $n\geq 3$,
\begin{align}
\left(L_{4}\right)_{g}=\left(-\Delta_{g}\right)^{2}+\delta_{g}\left(\left(4 P_{g}-(n-1) J_{g} g\right)(\nabla \cdot, \cdot)\right)+\frac{n-3}{2}\left(Q_{4}\right)_{g}
\end{align}
where $\delta$ denotes divergence, $\nabla$ denotes gradient on functions, $P_{g}$ the Schouten tensor $P_g=\frac{1}{n-1}\left(Ric_g-J_gg\right)$, $J_{g}=\frac{1}{2n} R_{g}, R_{g}$ is the scalar curvature of the metric $g$ and $Q_{4}$ is the $Q$-curvature
$$
\left(Q_{4}\right)_{g}=-\Delta_{g} J_{g}+\frac{n+1}{2} J_{g}^{2}-2\left|P_{g}\right|_{g}^{2} .
$$
In the following, we shall write $L_4=(L_4)_g$ for short when the background metric is understood.

When $n=3$, we have the following  conformal invariance property of $L_4$ and $Q_4$,
\begin{align}
\begin{split}\label{P4-conformal}
\left(L_{4}\right)_{\hat{g}} U&=e^{-4 \tau}\left(L_{4}\right)_{g}(U),\\
(Q_4)_{\hat g}&=e^{-4\tau}\left((L_4)_g\tau+(Q_4)_g\right),
\end{split}
\end{align}
for any smooth function $U$ on $X^4$ and $\hat g=e^{2\tau}g$.

When $X^4$ has boundary $M$, \citet{chang1997zeta} derived a conformally covariant boundary operator $P_{3}^{b}$ and associated $T$-curvature $T_3$. Suppose $h$ is the induced metric of $(X^4,g)$ on $M$. Let us use $\bar{\Delta}$ and $\bar\nabla$ denote the Laplacian and connection on $(M,h)$. Assume $A$ is the second fundamental form of $M$ and $H=tr_{h}A$ is the mean curvature. Then
\begin{align*}
P_{3}^b(u)&=-\frac{1}{2} \eta \Delta u-\bar{\Delta} \eta u+\left\langle A-\frac{2}{3} H h, \bar{\nabla}^{2} u\right\rangle+\frac{1}{3}\langle\bar{\nabla} H, \bar{\nabla} u\rangle-(\operatorname{Ric}(\eta, \eta)-2 J) \eta u\\
T_3&=\frac{1}{2} \eta J-\frac{1}{3} \bar{\Delta} H+J H-\langle R(\eta, \cdot, \eta, \cdot), A\rangle-\frac{1}{3} \operatorname{tr} A^{3}+\frac{1}{9} H^{3} .
\end{align*}
Here  $\eta=\eta_g$ is the outward-pointing unit normal to $M$.


If $\hat g=e^{2 \tau} g$, we have the transformation laws
\begin{align}
\begin{split}\label{P3b-conformal}
    \left(P_{3}^{b}\right)_{\hat g}U&=e^{-3 \tau}\left(P_{3}^{b}\right)_{g}U,\\
    \left(T_{3}\right)_{\hat g}&= e^{-3 \tau}\left(    \left(P_{3}^{b}\right)_{g} \tau+\left(T_{3}\right)_{g}\right).
    \end{split}
\end{align}
We also have the relation of mean curvature
\begin{align}\label{H-conformal}
H_{\hat g}=e^{-\tau}\left(H_g+n\eta_g\tau\right).
\end{align}

For the model case $(\mathbb{B}^4,\mathbb{S}^3,g_0)$ where $g_0$ is the Euclidean metric, one has (for instance, see \cite[(6.6)]{ache2017sobolev})
\begin{align}\label{geo-facts}
(P_{3}^{b})_{g_0}=-\frac{1}{2} \eta \Delta-\bar{\Delta} \eta-\bar{\Delta},\quad (T_{3})_{g_0}=2,\quad H_{g_0}=3.
\end{align}

On $\mathbb{B}^4$, there is a special metric $g^*=e^{1-|\xi|^2}g_0$ which has nice properties. It is called \textit{adapted metric} in \cite{case2016fractional} (also appeared in \cite{fefferman2002}). Under this metric $g_*$, $\mathbb{S}^3$ is totally geodesic and 
\begin{align}
&\left(Q_{4}\right)_{g^{*}}=\left(Q_{4}\right)_{g_{0}}=0 ,\\
&\left(T_{3}\right)_{g^{*}}=\left(T_{3}\right)_{g_{0}}=2.\label{T3-conformal}
\end{align}


Now we are ready to prove the main theorem of this subsection. 
\begin{proof}[Proof of Theorem \ref{thm:ext>=4} part (2)]
Suppose that $v$ is the extension. Then $v$ will be an extremal function for the \eqref{ache-main=4}. It is easy to see that the Euler-Lagrange equation of \eqref{ache-main=4} is 
\begin{align}\label{E-L-four}
\left\{\begin{array}{lll}
    \Delta^2 v=0& \text{in }\mathbb{B}^4,\\
    -\eta\Delta v-2\bar \Delta v+4=8\pi^2\left(\int_{\mathbb{S}^3}e^{3v}\right)^{-1}e^{3v}&\text{on }\mathbb{S}^4,\\
    \eta v=0&\text{on }\mathbb{S}^4.
    \end{array}
\right.
\end{align}
Here $\Delta=\Delta_{g_0}$ and $\eta=\eta_{g_0}$. 
Let $g=e^{2v}g^*=e^{2v+1-|\xi|^2}|d\xi|^2$, here $g^*$ is the so-called \textit{adapted metric}. Denote $\tau=v+(1-|\xi|^2)/2$.  The first line of \eqref{E-L-four}
implies $(P_4)_{g_0}\tau=\Delta_{g_0}^2\tau=\Delta_{g_0}^2v=0$. Then applying \eqref{P4-conformal} with $\tau=v+(1-|\xi|^2)/2$,  the first line of \eqref{E-L-four} is equivalent to 
\begin{align}
(Q_4)_{g}&=e^{-4\tau}((P_4)_{g_0}\tau+(Q_4)_{g_0})=0.
\end{align}
Applying \eqref{P3b-conformal} with $\tau=v+(1-|\xi|^2)/2$, the second line of \eqref{E-L-four} is equivalent to 
\begin{align}
(T_3)_{g}=e^{-3\tau}((P_3^b)_{g_0}\tau+(T_3)_{g_0})=e^{-3v}\left(-\frac{1}{2}\eta_{g_0}\Delta_{g_0}v-\bar{\Delta}_{g_0}v+2\right)=const>0
\end{align}
where we have used \eqref{geo-facts}, $\tau=v$ on $\mathbb{S}^4$ and $\eta_{g_0}v=0$. Applying \eqref{H-conformal} with the same $\tau$ as before, the third line of \eqref{E-L-four} is equivalent to 
\begin{align}
H_g=e^{-\tau}(H_{g_0}+3\eta_{g_0}\tau)=e^{-v}(3+3(\eta_{g_0}v-1))=0.
\end{align}
Combining the above analysis, \eqref{E-L-four} is equivalent to 
\begin{align}\label{Q-T-H}
    (Q_4)_g=0,\quad (T_3)_g=const>0,\quad H_g=0.
\end{align}
 Using M\"obius transformation \ref{def:S}, we can find $w$ such that $(\mathbb{B}^4\setminus\{(0,0,0,-1)\}, g)$ is isometric to $(\mathbb{R}^4_+,e^{2w}(|dx|^2+dt^2))$ through
\begin{align}
    \mathcal{S}^*(e^{2v+1-|\xi|^2}|d\xi|^2)=e^{2w}(|dx|^2+dt^2)
\end{align}
where $v$ and $w$ are related by 
\begin{align}
    w&=v\circ\mathcal{S}+\frac{1}{2}-\frac12\frac{|x|^2+(t-1)^2}{|x|^2+(t+1)^2}+\log \frac{2}{x^2+(1+t)^2},\label{w=vS}\\
    v&=w\circ \mathcal{S}^{-1}-\frac{1-|\xi|^2}{2}+\log \frac{2}{|\xi'|^2+(1+\xi_4 )^2}\label{v=wS}.
\end{align}
By the isometry, we can think of \eqref{Q-T-H} as referring to $\mathcal{S}^* g$ on $\mathbb{R}^4_+$. Thus using $|dx|^2+dt^2$ as the background metric and the conformal properties of $L_4$, $P^b_3$ and $H$, we rewrite \eqref{Q-T-H} as the following
\begin{align}\label{wR4}
\begin{cases}
    \Delta^2 w=0&\text{in }\mathbb{R}^4_+,\\
    \partial_t \Delta w=ce^{3w}&\text{on }\mathbb{R}^3,\\
    {\partial_t}w=0&\text{on }\mathbb{R}^3,
\end{cases}
\end{align}
for some constant $c>0$. Moreover, the isometry also implies
\begin{align}\label{fin-vol-R4}
    \int_{\mathbb{R}^4_+}e^{4w(x,t)}dxdt=\text{vol}(\mathbb{B}^4,g)<\infty,\quad \int_{\mathbb{R}^3_+}e^{3w(x,0)}dx=\text{vol}(\mathbb{S}^3,g|_{\mathbb{S}^3})<\infty.
\end{align}

The solution to \eqref{wR4} with \eqref{fin-vol-R4}  has been studied by Theorem \ref{thm:main-1}. 
Since $v$ is smooth on $\overline{\mathbb{B}^4}$, then \eqref{w=vS} leads to $w(x,t)=o(|x|^2+t^2)$ as $|x|+t\to \infty$. Therefore, it follows from Theorem \ref{thm:main-1} that there exists $\lambda>0$, $a\in \mathbb{R}^3$ and a constant $c$ such that  
\begin{align}
    w(x,t)=\log\frac{2\lambda}{(\lambda+t)^2+|x-a|^2}+\frac{2t\lambda}{(\lambda+t)^2+|x-a|^2}+c.
\end{align}
Plugging in \eqref{tx2xi} with $\omega_0=\mathcal{S}(a,\lambda)$ and \eqref{x=xi,t=xi4} to \eqref{v=wS}, one obtains
\begin{align*}
    v(\xi)=&\ \log F(\xi,\omega_0)^{-2}+\frac{(1-|\omega_0|^2)(1-|\xi|^2)}{2F(\xi,\omega_0)^2}-\frac{1}{2}(1-|\xi|^2)+c\\
    =&\ -\log{F(\xi,\omega_0)^2}+\frac{1-|\xi|^2}{2}\left[\frac{1-|\omega_0|^2}{F(\xi,\omega_0)^2}-1\right]+c
\end{align*}
The precise value of $c$ can be determined through $v(\xi)=-\log|1-\langle z_0,\xi\rangle|$. 
This completes the proof.
\end{proof}

\begin{remark}
The above method also applies to the harmonic case. The proof is simpler in that case because the adapted metric $g^*$ for $\mathbb{B}^{n+1}$ is identical to the Euclidean metric $g_0$ (see \cite[Rmk 2.4]{ache2017sobolev}).
\end{remark}

\section{Classification of the solution to a bi-harmonic equation}
In this section, we will prove Theorem \ref{thm:main-1}. The strategy is to separate the nonlinear effect, by subtracting a function constructed from nonlinear boundary conditions. Such trick has been used by  \cite{jin2015existence,sun2016classification}. The resulting linear fourth order equation can be classified under the finite volume condition. The proof here is greatly inspired by  \citet{lin1998classification}, who initiated the classification of some conformal bi-harmonic equation on $\mathbb{R}^4$.

Given any $f\in L^1(\mathbb{R}^3)$,  we define $v$ for $(x,t)\in \mathbb{R}^4_+$
\begin{align}
    v(x,t)=\frac{1}{|\mathbb{S}^3|}\int_{\R^3}f(y)\log\frac{|y|^2}{|x-y|^2+t^2}dy.
\end{align}
\begin{lemma}
For any $f(y)\in C^2(\R^3)\cap L^1(\R^3)$, one has 
\begin{align}
\begin{cases}
    \Delta^2v=0&\text{in}\quad\mathbb{R}^4_+,\\
    \partial_t\Delta v=4f&\text{on}\quad \partial \mathbb{R}^4_+,\\
    \partial_t v=0&\text{on}\quad \partial \mathbb{R}^4_+.
\end{cases}
\end{align}
\end{lemma}
\begin{proof}
Using the Lebesgue dominating theorem, it is easy to see $\partial_t v(x,0)=0$ and for any $t>0$
\begin{align}
    \Delta v(x,t)=&\ \frac{-4}{|\mathbb{S}^3|}\int_{\R^3}\frac{f(y)}{|x-y|^2+t^2}dy\\
    \partial_t\Delta v(x,t)=&\ \frac{8}{|\mathbb{S}^3|}\int_{\mathbb{R}^3}\frac{t f(y)}{(|x-y|^2+t^2)^2}dy\\
    \Delta^2v=&\ 0.
\end{align}
Note that $\frac{2}{|\mathbb{S}^3|}\frac{t}{(|x-y|^2+t^2)^2}$ is the Possion kernel of $\Delta$ on $\R^4_+$. Then one has 
\begin{align}
\lim_{t\to 0^+}\partial_t\Delta v(x,t)=4f(y).
\end{align}

\end{proof}
Now suppose that $u$ satisfies the assumptions of Theorem \ref{thm:main-1}. In the following, we will denote
\begin{align}\label{def:v}
    v(x,t)=\frac{1}{|\mathbb{S}^3|}\int_{\R^3}e^{3u(y,0)}\log\frac{|y|^2}{|x-y|^2+t^2}dy.
\end{align}
\begin{lemma}\label{lem:v}
    For $v(x,t)$ defined in \eqref{def:v}, there exists some constant $C>0$ such that 
    \begin{align}
       v(X)\geq  -\alpha \log|X|-C
    \end{align}
    where
    \begin{align}\label{def:alpha}
        \alpha=\frac{2}{|\mathbb{S}^3|}\int_{\R^3}e^{3u(y,0)}dy.
    \end{align}
\end{lemma}
\begin{proof}
    The proof is essentially contained in \cite{lin1998classification}. For readers' convenience, we present it here.

    For $|x|\geq 4$, we decompose $\mathbb{R}^3=A_1\cup A_2$, where $A_1=\{y||(y,0)-X|\leq |X|/2\}$ and $A_2=\{y||(y,0)-X|\geq |X|/2\}$. For $y\in A_1$, one has $|y|\geq |X|-|X-(y,0)|\geq |X|/2\geq |X-(y,0)|$. Consequently, we have $\log |y|/|X-(y,0)|\geq 0$ and 
    \begin{align}
    \int_{A_1}e^{3u(y,0)}\log\frac{|y|^2}{|x-y|^2+t^2}dy\geq 0.
    \end{align}

    For $y\in A_2$, one has $|X-(y,0)|\leq |X||y|$ if $|y|\geq 2$ and $\log|X-(y,0)|\leq \log |X|+C$ if $|X|\geq 4$ and $|y|\leq 2$. Thus
    \begin{align}\label{D2v-eqn}
        v(x)\geq&\  \frac{1}{|\mathbb{S}^3|}\int_{A_2}e^{3u(y,0)}\log\frac{|y|^2}{|x-y|^2+t^2}dy\\
        \geq &\ -\frac{2}{|\mathbb{S}^3|}\log |X|\int_{A_2}e^{3u(y,0)}dy+\frac{1}{|\mathbb{S}^3|}\int_{|y|\leq 2}e^{3u(y,0)}\log\frac{|y|^2}{|X-(y,0)|^2}dy\\
        \geq &\ -\alpha \log |X|-C.
    \end{align}

\end{proof}
\begin{lemma}\label{lem:deltau}
Suppose $u$ satisfies \eqref{main-eq} and \eqref{finite-vol}. Then there exists a constant $C_1\geq 0$ such that
\begin{align}
    \Delta u(x,t)=-\frac{4}{|\mathbb{S}^3|}\int_{\R^3}\frac{e^{3u(y,0)}}{|x-y|^2+t^2}dy-C_1.
\end{align}
Moreover, there exist constants $c_*\leq 0$, $a_i\leq 0$, $i=1,2,3$ such that  
\begin{align}\label{u-v=t+p}
    u-v=c_*t^2+\sum_{i=1}^3a_i(x_i-x_i^0)^2+c_0.
\end{align}
\end{lemma}
\begin{proof}
Denote $w=u-v$ where $v$ is defined in \eqref{def:v}. Then $w$ satisfies
\begin{align}
\begin{cases}
    \Delta^2w=0&\text{in}\quad\mathbb{R}^4_+,\\
    \partial_t\Delta w=0&\text{on}\quad \partial \mathbb{R}^4_+,\\
    \partial_t w=0&\text{on}\quad \partial \mathbb{R}^4_+.
\end{cases}
\end{align}
We extend $w$ by $w(x,t)=w(x,-t)$ for $t<0$.  We denote $\hat w$ this new function on $\mathbb{R}^4$. It follows from the mean value property of harmonic functions that $\Delta \hat w$ is smooth is $\R^4$. Consequently, $\hat w$ is also smooth in $\R^4$.

It follows from Lemma \ref{lem:v} that for $t>0$
\begin{align}
w(x,t)= u(x,t)-v(x,t)\leq u(x,t)+\alpha \log|(x,t)|+C. 
\end{align}
Thus $\hat w(X)\leq \hat u(X)+\alpha \log |X|+C$, where $\hat u$ is the even extension of $u$ to $\R^4$. By Pizzetti's formula (see \cite{martinazzi2009classification}), we have
\begin{align}\label{pizzetti}
    \frac{r^2}{8}\Delta\hat w(X_0)=\fint_{\partial B_r(X_0)}\hat wd\sigma-\hat w(X_0).
\end{align}
By Jensen's inequality
\begin{align}
    \exp\left(\frac{r^2}{2}\Delta\hat w(X_0)\right)&\leq e^{-4\hat w(X_0)}\exp\left(4\fint_{\partial B_{r}(X_0)}\hat w d\sigma\right)\\
    &\leq e^{-4\hat w(X_0)}\fint_{\partial B_r(X_0)}e^{4\hat w}d\sigma.
\end{align}
Since $\hat w(X)\leq \hat u(X)+\alpha \log |X|+C$ and \eqref{finite-vol-in}, then  $r^{3-4\alpha}\exp(\frac{r^2}{2}\Delta\hat w(X_0))\in L^1[1,\infty)$. Thus $\hat w(X_0)\leq 0$ for all $x_0\in \R^4$. By Liouville's Theorem, $\Delta \hat w(X)\equiv - C_1$ in $\R^4$ for some constant $C_1\geq 0$.

For bi-harmonic functions, one has the following fact that (for instance, see \cite{martinazzi2009classification})
\begin{align}
    |D^3\hat w|(X_0)\leq  \frac{C}{r^3}\fint_{B_r(X_0)}\hat w d\sigma
\end{align}
holds for some universal constant $C$.
However, \eqref{pizzetti} and $\Delta \hat w(X)=-C_1$ implies that 
\begin{align}
    \fint_{B_r(X_0)}\hat w d\sigma =O(r^2)
\end{align}
Thus $|D^3 \hat w|(X_0)=0$. Therefore $\hat w$ is a polynomial of degree at most 2. By the boundary condition of $w$ and even symmetry of $\hat w$, one has $\hat w=c_*t^2+p(x)$ where $p(x)$ has degree at most 2. 

Since $ w(X)\leq  u(X)+\alpha \log |X|+C$ for $X\in \R^4_+$ and $u$ satisfies \eqref{finite-vol-in}, then $c_*<0$. Moreover, after an orthorgonal transformation, we can assume $p(x)=\sum_{i=1}^3a_{i}x_i^2+b_ix_i+c_0$. Since $\int_{\R^3}e^{3u(x,0)}dx<\infty$, then we must have $a_i\leq 0$ and  $b_i=0$ whenever $a_i=0$. Thus 
\begin{align}
    p(x)=\sum_{i}a_i(x-x_i^0)^2+c_0.
\end{align}
The proof is complete.
\end{proof}
\bigskip
\begin{lemma}\label{lem:boundary}
Suppose that $u$ satisfies the assumptions of Theorem \ref{thm:main-1}. If  $u(x,0)=o(|x|^2)$ or $\bar\Delta u(x,0)=o(1)$ as $|x|\to \infty$,
then there exists $\lambda>0$ and $a\in \R^3$ such that
\begin{align}\label{ux0=log}
    u(x,0)=\log\left(\frac{2\lambda}{\lambda^2+|x-a|^2 }\right).
\end{align}
Consequently there exist some $c\leq 0$ such that 
\begin{align}\label{poisson-ext}
    u(x,t)=\frac{4}{\pi^2}\int_{\R^3}\frac{t}{(|x-y|^2+t^2)^3}\log\left(\frac{2\lambda}{\lambda^2+|y-a|^2 }\right)dx+ct^2
\end{align}
\end{lemma}
\begin{proof}
    Applying Lemma \ref{lem:deltau} and letting $t\to 0^+$ in \eqref{u-v=t+p}, we have 
    \begin{align}
    u(x,0)=\frac{1}{|\mathbb{S}^3|}\int_{\R^3}e^{3u(y,0)}\log\frac{|y|}{|x-y|}dy+\sum_{i=1}^3a_i(x_i-x_i^0)^2+c_0.
    \end{align}
    If $u(x)=o(|x|^2)$ or $\Delta u(x,0)=o(1)$, then $a_1=a_2=a_3= 0$. Thus 
    \begin{align}\label{ux0=}
    u(x,0)=\frac{1}{|\mathbb{S}^3|}\int_{\R^3}e^{3u(y,0)}\log\frac{|y|}{|x-y|}dy+c_0.
    \end{align}
    The solutions to such equation have been classified by \citet{xu2005uniqueness}. More precisely, $u$ must take the form \eqref{ux0=log} for some $a\in \R^3$ and $\lambda>0$.

    Once we know the boundary data, then $v(x,t)$ defined in \eqref{def:v} satisfies $\Delta^2 v=0$ and $\partial_tv(x,0)=0$, then using the Poisson kernel of $\Delta^2$ on $\mathbb{R}^4_+$ (see \cite[Lemma 2.2]{sun2016classification}), one knows that $v(x,t)$ takes the form \eqref{poisson-ext}.
    \begin{align}\label{v-explicit}
    v(x,t)=\frac{4}{\pi^2}\int_{\R^3}\frac{t}{(|x-y|^2+t^2)^3}\log\left(\frac{2\lambda}{\lambda^2+|y-a|^2 }\right)dx.
    \end{align}
    Then $u(x,t)$ takes the form \eqref{poisson-ext} by \eqref{u-v=t+p} 
    
\end{proof}

It seems hard to integrate \eqref{poisson-ext} out explicitly. We get around this difficulty by constructing a solution to \eqref{main-eq} directly.


\begin{lemma}\label{lem:ual}
 For any $\lambda>0$ and $a\in \mathbb{R}^3$, \eqref{def:u}
satisfies \eqref{main-eq}.
\end{lemma}
\begin{proof}
We only prove the lemma for $a=0$ and $\lambda=1$. Let $u=u_1+u_2$
\begin{align}
    u_1=\log\left(\frac{2}{(1+ t)^2+|x|^2 }\right),\quad u_2=\frac{2t}{(1+t)^2+|x|^2}.
\end{align}
Then it is easy to see 
\begin{align}
    \partial_t u_1(x,0)=-\frac{2}{1+|x|^2}=-\partial_t u_2(x,0)
\end{align}
Therefore $\partial_tu(x,0)=0$ for $x\in \mathbb{R}^3$. We continue taking derivatives
\begin{align}\label{Lap-u}
    \Delta u_1=-\frac{4}{(1+t)^2+|x|^2},\quad \Delta u_2=-\frac{8 (t+1)}{\left((1+t)^2+|x|^2\right)^2}
\end{align}
\begin{align}
    \partial_t\Delta u_1(x,0)=\frac{8}{\left(1+|x|^2\right)^2},\quad \partial_t \Delta u_2(x,0)=-\frac{8 \left(|x|^2-3\right)}{\left(1+|x|^2\right)^3}.
\end{align}
Therefore
\begin{align}
    \partial_t\Delta u(x,0)=\frac{32}{\left(1+|x|^2\right)^3}=4e^{3u(x,0)}.
\end{align}
Last, we have
\begin{align}
    \Delta^2u_1(x,t)=\Delta^2 u_2(x,t)=0
\end{align}
This implies $\Delta^2 u=0$ in $\mathbb{R}^4$. Thus $u$ satisfies \eqref{main-eq}. It is easy to see that \eqref{finite-vol} holds.
\end{proof}


\begin{proof}[Proof of Theorem \ref{thm:main-1}] If $\bar\Delta u(x,0)=o(1)$ as $|x|\to \infty$, then Lemma \ref{lem:boundary} implies that $u(x,t)=v(x,t)+ct^2$ where $v$ is defined in \eqref{v-explicit}. Such $v$ is uniquely determined and satisfies \eqref{D2v-eqn}. However, in Lemma \ref{lem:ual} we find $u_{a,\lambda}$ also satisfies the same equation. Therefore $v=u_{a,\lambda}$. The proof is complete.
\end{proof}

\begin{proof}[Proof of Corollary \ref{corcla}]
Taking $t=0$ in \eqref{u=int} implies that $u$ satisfies \eqref{ux0=}, which leads to \eqref{ux0=log}. Then Lemma \ref{lem:boundary} holds with $c=0$ in \eqref{poisson-ext}. The result follows from Lemma \ref{lem:ual}.
\end{proof}

\section*{Acknowledgement} L. Sun was partially supported by CAS Project for Young Scientists in Basic Research, Grant No. YSBR$\text{-}$031. C. B. Ndiaye was partially supported by NSF grant DMS--2000164. We would like to thank Jingang Xiong for helpful suggestions.

\bibliographystyle{plainnat}
\bibliography{ref.bib}

\end{document}